\documentclass[a4paper,11pt]{article}

\usepackage[defaultlines=5,all]{nowidow}		
\usepackage{indentfirst}
\usepackage[T1]{fontenc}
\usepackage[utf8]{inputenc}
\usepackage{authblk}
\usepackage{diagbox}
\usepackage{amsthm,amsmath}
\usepackage{tikz}
\usepackage{paralist}
\usepackage[shortlabels]{enumitem}
\usepackage{subcaption}
\usepackage{hyperref}		
\usepackage{stfloats}
\usepackage{verbatim}
\usepackage{algorithm}
\usepackage{algorithmic}	
\usepackage{hhline}			
\hypersetup{
    colorlinks=true,
    citecolor = blue,
    linkcolor=blue,
    filecolor=magenta,
    urlcolor=cyan,
}

\newtheorem{theorem}{Theorem}
\newtheorem{lemma}{Lemma}
\newtheorem{claim}{Claim}

\title{{\bf Critical $(P_5,dart)$-Free Graphs}}
\author[a,b]{Wen Xia}
\author[c]{Jorik Jooken}
\author[c,d]{Jan Goedgebeur}
\author[,a,b]{Shenwei Huang\thanks{Email: shenweihuang@nankai.edu.cn.}}

\affil[a]{College of Computer Science, Nankai University, Tianjin 300071, China}
\affil[b]{Tianjin Key Laboratory of Network and Data Security Technology, Nankai University, Tianjin 300071, China}
\affil[c]{Department of Computer Science, KU Leuven Campus Kulak-Kortrijk, 8500 Kortrijk, Belgium}
\affil[d]{Department of Applied Mathematics, Computer Science and Statistics, Ghent University, 9000 Ghent, Belgium}

\date{September 2, 2023}

\begin{document}

\maketitle

\begin{abstract}
Given two graphs $H_1$ and $H_2$, a graph is $(H_1,H_2)$-free if it contains
no induced subgraph isomorphic to $H_1$ nor $H_2$. A dart is the graph obtained from a diamond by adding a new vertex and making
it adjacent to exactly one vertex with degree 3 in the diamond.

In this paper, we show that there are finitely many $k$-vertex-critical
$(P_5,dart)$-free graphs for $k \ge 1$.
To prove these results, we use induction
on $k$ and perform a careful structural analysis via Strong Perfect Graph Theorem combined with the pigeonhole principle based on the properties of vertex-critical graphs.
Moreover, for $k \in \{5, 6, 7\}$ we characterize all $k$-vertex-critical $(P_5,dart)$-free graphs using a computer generation algorithm.
Our results imply the existence of a polynomial-time certifying algorithm to decide the $k$-colorability of
$(P_5,dart)$-free graphs for $k \ge 1$ where the certificate is either a $k$-coloring or a $(k+1)$-vertex-critical induced subgraph.


\end{abstract}

\section{Introduction}

All graphs in this paper are finite and simple.
A \emph{$k$-coloring} of a graph $G$ is a function $\phi:V(G)\longrightarrow \{ 1, \ldots ,k\}$ such that
$\phi(u)\neq \phi(v)$ whenever $uv \in E(G)$. Equivalently, a $k$-coloring of $G$ can be viewed
as a partition of $V(G)$ into $k$ stable sets. If a $k$-coloring exists, we say that $G$ is {\em $k$-colorable}.
The \emph{chromatic number} of $G$, denoted by
$\chi (G)$, is the minimum number $k$ such that $G$ is $k$-colorable.
A graph $G$ is {\em $k$-chromatic} if $\chi(G)=k$. A graph $G$ is {\em $k$-critical} if it is
$k$-chromatic and $\chi(G-e)<\chi(G)$ for any edge $e\in E(G)$.  For instance, $K_2$ is the
only 2-critical graph and odd cycles are the only 3-critical graphs.
A graph is {\em critical} if it is $k$-critical for some integer $k\ge 1$.
Vertex-criticality is a weaker notion. A graph $G$ is $k$-vertex-critical if $\chi(G) = k$ and $\chi(G - v) < k$ for any $v \in V(G)$.

 For a fixed $k\ge 3$, it has long been known that determining the $k$-colorability
 of a general graph is an NP-complete problem~\cite{K72}.
 However, the situation changes if one restricts the structure of
 the graphs under consideration.

Let $\mathcal{H}$ be a set of graphs. A graph $G$ is $\mathcal{H}$-free if
it does not contain any member in $\mathcal{H}$ as an induced subgraph.
When $\mathcal{H}$ consists of a single graph $H$ or two graphs $H_1$ and
$H_2$, we write $H$-free and $(H_1,H_2)$-free instead of $\{H\}$-free and $\{H_1,H_2\}$-free, respectively.
We say that $G$ is $k$-vertex-critical $\mathcal{H}$-free if it is $k$-vertex-critical
and $\mathcal{H}$-free. In this paper, we study $k$-vertex-critical $\mathcal{H}$ -free graphs. The following problem arouses our interest: Given a set $\mathcal{H}$
of graphs and an integer $k \geq 1$, are there finitely many $k$-vertex-critical $\mathcal{H}$-free graphs? This question is very important because
the finiteness of the set has a fundamental algorithmic implication.

	\begin{theorem}[Folklore]\label{Folklore}
		If the set of all $k$-vertex-critical $\mathcal{H}$-free graphs is finite, then there is a polynomial-time algorithm to determine whether an $\mathcal{H}$-free graph is $(k-1)$-colorable.  \qed
	\end{theorem}

Let $K_n$ be the complete graph on $n$ vertices. Let $P_t$ and $C_t$ denote the path and the cycle
on $t$ vertices, respectively. The complement of $G$ is denoted by $\overline{G}$. For two graphs
$G$ and $H$, we use $G+H$ to denote the disjoint union of $G$ and $H$. For a positive integer, we use $rG$
to denote the disjoint union of $r$ copies of $G$. For $s,r \geq 1$, let $K_{r,s}$ be the complete bipartite
graph with one part of size $r$ and the other part of size $s$. Our research is mainly motivated by  the
following two theorems.

\begin{theorem}[\cite{HMRSV15}]
For any fixed integer $k \geq 5$, there are infinitely many $k$-vertex-critical $P_5$-free graphs.
\end{theorem}

It is natural to consider which subclasses of $P_5$-free graphs have finitely many $k$-vertex-critical graphs.
In 2021, Cameron, Goedgebeur, Huang and Shi~\cite{CGHS21} obtained the following dichotomy result.

\begin{theorem}[\cite{CGHS21}]
		Let $H$ be a graph of order 4 and $k\ge 5$ be a fixed integer. Then there are infinitely many k-vertex-critical $(P_5,H)$-free graphs if and only if $H$ is $2P_2$ or $P_1+K_3$.
	\end{theorem}

This theorem completely solves the finiteness problem of $k$-vertex-critical $(P_5,H)$-free graphs for $|H|=4$. In~\cite{CGHS21}, the authors also posed the
natural question of which five-vertex graphs $H$ lead to finitely many $k$-vertex-critical $(P_5,H)$-free graphs.

It is known that there are exactly
13 5-vertex-critical $(P_5,C_5)$-free graphs \cite{HMRSV15}. Recently, Cameron and Ho\`{a}ng constructed infinite families of $k$-vertex-critical $(P_5,C_5)$-free graphs for $k \geq 6$~\cite{CH23}. It has been proven that there are finitely many $k$-vertex-critical $(P_5,banner)$-free graphs for $k=5$~\cite{HLS19} and 6~\cite{CHLS19} and there are finitely many $k$-vertex-critical $(P_5,\overline{P_5})$-free graphs for fixed $k$~\cite{DHHMMP17}. Hell and Huang proved that there are finitely many
$k$-vertex-critical $(P_6,C_4)$-free graphs~\cite{HH17}. This was later generalized to $(P_t,K_{r,s})$-free graphs in the context of $H$-coloring~\cite{KP17}. This gives
an affirmative answer for $H=K_{2,3}$. In~\cite{CGS}, Cai, Goedgebeur and Huang showed that there are finitely many $k$-vertex-critical $(P_5,gem)$-free
graphs and finitely many $k$-vertex-critical $(P_5,\overline{P_3+P_2})$-free graphs.
Later, Cameron and Hoa\`ng \cite{CH23gem} gave a better bound
on the order of $k$-vertex-critical $(P_5,gem)$-free graphs and determined all such graphs for $k\le 7$.
Moreover, it has been proven that that there are finitely many
5-vertex-critical $(P_5,bull)$-free graphs~\cite{HLX23} and finitely many 5-vertex-critical $(P_5,chair)$-free graphs~\cite{HL23}.

\noindent {\bf Our contributions.} A dart (see \autoref{dart}) is the graph obtained from a diamond by adding a new vertex and making
it adjacent to exactly one vertex with degree 3 in the diamond. Our main result is as follows.

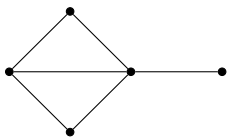
\begin{figure}[h]
		\centering
		\begin{tikzpicture}[scale=0.8]
			\tikzstyle{vertex}=[draw, circle, fill=black!100, minimum width=1pt,inner sep=1pt]
			
			\node[vertex](v1) at (-1,0) {};
			\node[vertex](v2) at (0,1) {};
			\node[vertex](v3) at (1,0) {};
			\node[vertex](v4) at (0,-1){};
			\node[vertex](v5) at (2.5,0) {};
			\draw (v1)--(v2)--(v3)--(v4)--(v1) (v1)--(v3) (v3)--(v5);	
		\end{tikzpicture}
		\caption{The dart graph.}
		\label{dart}
	\end{figure}

\begin{theorem}\label{th:6-vertex-critical}
   For every fixed integer $k \ge 1$, there are finitely many $k$-vertex-critical $(P_5, dart)$-free graphs.
\end{theorem}

We prove a Ramsey-type statement (see \autoref{ST}) which allows us to prove our main result by induction on $k$.
This is another example of a result that is proved by induction on $k$ for the finitess problem besides the one in~\cite{DHHMMP17}.

We perform a careful structural analysis via Strong Perfect Graph Theorem combined with the pigeonhole principle based on the properties of vertex-critical graphs.
Moreover, for $k \in \{5, 6, 7\}$ we computationally determine a list of all $k$-vertex-critical $(P_5,dart)$-free graphs.

Our results imply the existence of polynomial-time certifying algorithm to decide the $k$-colorability of
$(P_5,dart)$-free graphs for $k \ge 1$.
(An algorithm is \textit{certifying} if, along with the answer given by the algorithm, it also gives a certificate which allows to verify in polynomial time that the output of the algorithm is indeed correct; in case of $k$-coloring the the certificate is either a $k$-coloring or a $(k+1)$-vertex-critical induced subgraph.)

\begin{theorem}
For every fixed integer $k\ge 1$, there is a polynomial-time certifying algorithm to decide the $k$-colorability of
$(P_5,dart)$-free graphs.
\end{theorem}

\begin{proof}
Let $G$ be a $(P_5,dart)$-free graph. We first run the polynomial-time algorithm for determining whether a $P_5$-free graph
is $k$-colorable from~\cite{HKLSS10} for $G$. If the answer is yes, the algorithm outputs a $k$-coloring of $G$.
Otherwise $G$ is not $k$-colorable. In this case, $G$ must contain a $(k+1)$-vertex-critical $(P_5,dart)$-free graph
as an induced subgraph.  For each $(k+1)$-vertex-critical $(P_5,dart)$-free graph $H$, it takes polynomial-time to check
whether $G$ contains $H$. Since there are only finitely many such graphs by \autoref{th:6-vertex-critical}, we can do this for
every such graph $H$ and the total running time is still polynomial.
\end{proof}

	The remainder of the paper is organized as follows. We present some preliminaries in Section~\ref{Preliminarlies} and give structural properties around an induced $C_5$ in a ($P_5$,dart)-free graph in Section~\ref{structure}. We show that there are finitely many $k$-vertex-critical ($P_5$,dart)-free graphs for all $k \ge 1$ in Section~\ref{6-vertex-critical} and computationally determine an exhaustive list of such graphs for $k \in \{5,6,7\}$ in Section~\ref{algorithmSection}. Finally, we give a conclusion in Section~\ref{conclusion}.

\section{Preliminaries}\label{Preliminarlies}
For general graph theory notation we follow~\cite{BM08}. For $k\ge 4$, an induced  cycle of length $k$ is called a {\em $k$-hole}. A $k$-hole is an {\em odd hole} (respectively {\em even hole}) if $k$ is odd (respectively even). A {\em $k$-antihole} is the complement of a $k$-hole. Odd and even antiholes are defined analogously.
	
	Let $G=(V,E)$ be a graph. If $uv\in E(G)$, we say that $u$ and $v$ are {\em neighbors} or {\em adjacent}, otherwise $u$ and $v$ are {\em nonneighbors} or  {\em nonadjacent}. The {\em neighborhood} of a vertex $v$, denoted by $N_G(v)$, is the set of neighbors of $v$. For a set $X\subseteq V(G)$, let $N_G(X)=\cup_{v\in X}N_G(v)\setminus X$. We shall omit the subscript whenever the context is clear. For $x \in V(G)$ and $S\subseteq V(G)$, we denote by $N_S(x)$ the set of neighbors of $x$ that are in $S$, i.e., $N_S(x)=N_G(x)\cap S$. For two sets $X,S\subseteq V(G)$, let $N_S(X)=\cup_{v\in X}N_S(v)\setminus X$.

For $X,Y\subseteq V(G)$, we say that $X$ is {\em complete} (resp. {\em anticomplete}) to $Y$ if every vertex in $X$ is adjacent (resp. nonadjacent) to every vertex in $Y$. If $X=\{x\}$, we write ``$x$ is complete (resp. anticomplete) to $Y$'' instead of ``$\{x\}$ is complete (resp. anticomplete) to $Y$''. If a vertex $v$ is neither complete nor anticomplete to a set $S$, we say that $v$ is {\em mixed} on $S$. For a vertex $v\in V$ and an edge $xy\in E$, if $v$ is mixed on $\{x,y\}$, we say that $v$ is {\em mixed} on $xy$. For a set $H\subseteq V(G)$, if no vertex in $V(G) \setminus H$ is mixed on $H$, we say that $H$ is a {\em homogeneous set}, otherwise $H$ is a {\em nonhomogeneous set}.

 A vertex subset $S\subseteq V(G)$ is {\em stable} if no two vertices in $S$ are adjacent. A {\em clique} is the complement of a stable set. Two nonadjacent vertices $u$ and $v$ are said to be {\em comparable} if $N(v)\subseteq N(u)$ or $N(u)\subseteq N(v)$.
 For an induced subgraph $A$ of $G$, we write $G-A$ instead of $G-V(A)$. For $S\subseteq V$, the subgraph \emph{induced} by $S$ is denoted by $G[S]$. We say that a vertex $w$ {\em distinguishes} two vertices $u$ and $v$ if $w$ is adjacent to exactly one of $u$ and $v$.

 We proceed with a few useful results that will be used later. The first folklore property of vertex-critical graph is that such graphs
 contain no comparable vertices. A generalization of this property was presented in~\cite{CGHS21}.

     \begin{lemma}[\cite{CGHS21}] \label{lem:XY}
		Let $G$ be a $k$-vertex-critical graph. Then $G$ has no two nonempty disjoint subsets $X$ and $Y$ of $V(G)$ that satisfy all the following conditions.
		\begin{itemize}
			\item $X$ and $Y$ are anticomplete to each other.
			\item $\chi(G[X])\le\chi(G[Y])$.
			\item Y is complete to $N(X)$.
		\end{itemize}
	\end{lemma}

\begin{lemma}[\cite{HLX23}]\label{homo}
	Let $G$ be a 5-vertex-critical $P_5$-free graph and $S$ be a homogeneous set of $V(G)$. For each component $A$ of $G[S]$,
	
	\begin{enumerate}[(i)]
		\item if $\chi(A)=1$, then $A$ is a $K_1$;
		\item if $\chi(A)=2$, then $A$ is a $K_2$;
		\item if $\chi(A)=3$, then $A$ is a $K_3$ or a $C_5$.
	\end{enumerate}
\end{lemma}

We extend \autoref{homo} to all critical graphs, which may be of independent interest.

\begin{lemma}\label{lem:homoge}
Let $G$ be a $k$-vertex-critical graph and $S$ be a homogeneous set of $V(G)$.
For each component $A$ of $G[S]$, if $\chi(A) = m$ with $m < k$, then $A$ is an $m$-vertex-critical graph.
\end{lemma}

\begin{proof}
Suppose not. Then there exist some vertex $u \in V(A)$ such that $\chi(A-u) = m$. Since $G$ is
$k$-vertex-critical, $G-u$ is $(k-1)$-colorable. Let $\phi$ be a $(k-1)$-coloring of $G-u$ and
$\phi(A-u)=\{1,2,\ldots,s\}$ with $m \le s \le k-1$. Since $A$ is homogeneous, then $\phi(N(A-u))\notin \{1,2,\ldots,s\}$.
Let $f$ be an $m$-coloring  of $A$ and $f(A)= \{1,2,\ldots,m\}$. Let $g(A) = f$ and $g(G-A)= \phi|_{G-A}$. Then $g$ is a $(k-1)$-coloring of $G$. This contradicts $G$ is $k$-vertex-critical.

Therefore, $A$ is an $m$-vertex-critical graph.
\end{proof}

 Next, we prove an important lemma, which will be used frequently in the proof of our results.

\begin{lemma}\label{ST}
Let $G$ be a $dart$-free graph and $c$ be a fixed integer. Let $S,T$ be two disjoint subsets of $V(G)$ such that
$|S| \le c$ and $\overline{G[T \cup S]}$ is connected.
If every vertex in $S$ is adjacent to at least one of $w_1$ and $w_2$ where $w_1$ and $w_2$ are two arbitrary nonadjacent vertices in $T$,
and there exists a vertex $w \in V(G)\setminus{(S\cup T)}$ such that $w$ is complete
to $S\cup T$, then $|T| \le c({\chi(T))}^1 + \ldots + c({\chi(T)})^{2\chi(T)+1}$.
\end{lemma}

\begin{proof}
Let $N_0 = S$, $N_i = \{v|v \in T\setminus {\cup_{j=0}^{i-1}N_j}$, $v $ has a nonneighbor in $N_{i-1}\}$
       where $i \ge 1$. Let $u \in N_{i-1}$ and $v,v' \in N_{i}$ with $vv' \notin E(G)$.
       Since $\overline{G[T \cup S]}$ is connected, $N_i \neq \emptyset$.

        First, we show that $u$ is adjacent to at least one of $v$ and $v'$. If $i=1$, we are done.  Now we consider the case
       of $i \ge 2$. Suppose that $u$ is adjacent to neither $v$ nor $v'$. Let $u' \in N_{i-2}$ be the nonneighbor of $u$. Then $u'$ is complete to $v$ and $v'$ by the definition of $N_i$. Since $w$ is complete to $S\cup T$.
       Then  $\{v,v',u',w,u\}$ induces a dart. So $u$ is adjacent to at least one of $v$ and $v'$.

       Next, we show that $|N_i| \le c({\chi(T))}^{i}$.
       We have $\chi(N_i) \le \chi(T)$. Hence, we can partition $N_i$ into at most $\chi(T)$ stable sets. Because $u$ is adjacent to at least one of $v$ and $v'$, each vertex in $N_{i-1}$ has at most $\chi(T)$ nonneighbors in $N_i$. Since $|N_0| = |S| \le c$, we have $|N_i| \le c({\chi(T))}^{i}$.
 
       Finally, we show that $i \le 2\chi(T) +1$. Suppose not. By the definition of $N_i$, $N_i$ and $N_j$ are complete when $|i-j|>1$.
       Then we take a vertex $u_i \in N_i$ where $i$ is even. Now, $\{u_2,u_4,\ldots,u_{2\chi(T)+2}\}$ induces a $K_{\chi(T)+1}$, a contradiction.
       This shows $i \le 2\chi(T) +1$.

       Therefore, $|T| \le c({\chi(T))}^1 + \ldots + c({\chi(T)})^{2\chi(T)+1}$.
       \end{proof}

       The following theorem tells us there are finitely many 4-vertex-critical $P_5$-free graphs.

\begin{theorem}[\cite{BHS09,MM12}]\label{4-vertex-cri}
If $G = (V,E)$ is a 4-vertex-critical $P_5$-free graph, then $|V| \le 13$.
\end{theorem}

A property on bipartite graphs is shown as follows.

\begin{lemma}[\cite{F93}]\label{2K2}
		Let $G$ be a connected bipartite graph. If $G$ contains a $2K_2$, then $G$ must contain a $P_5$.
	\end{lemma}

 The clique number of $G$, denoted by $\omega(G)$, is the size of a largest clique in $G$. A graph $G$ is {\em perfect} if $\chi(H)=\omega(H)$ for every induced subgraph $H$ of $G$. Another result we use is the famous Strong Perfect Graph Theorem.
	
	\begin{theorem}[The Strong Perfect Graph Theorem~\cite{CRST06}]\label{thm:SPGT}
		A graph is perfect if and only if it contains no odd holes or odd antiholes.
	\end{theorem}

\section{Structure around 5-hole}\label{structure}

Let $G=(V,E)$ be a graph and $H$ be an induced subgraph of $G$.
We partition $V\setminus V(H)$ into subsets with respect to $H$ as follows:
for any $X\subseteq V(H)$, we denote by $S(X)$ the set of vertices
in $V\setminus V(H)$ that have $X$ as their neighborhood among $V(H)$, i.e.,
\[S(X)=\{v\in V\setminus V(H): N_{V(H)}(v)=X\}.\]
For $0\le m\le |V(H)|$, we denote by $S_m$ the set of vertices in $V\setminus V(H)$ that have exactly $m$
neighbors in $V(H)$. Note that $S_m=\bigcup_{X\subseteq V(H): |X|=m}S(X)$.

Let $G$ be a $(P_5,dart)$-free graph and $C=v_1,v_2,v_3,v_4,v_5$ be an induced $C_5$ in $G$.
We partition $V\setminus C$ with respect to $C$ as follows, where all indices below are modulo five.

$S_0=\{v \in V\backslash V(C): N_C(v)=\emptyset\},$

$S_2(i)=\{v \in V\backslash V(C): N_C(v)=\{v_{i-1},v_{i+1}\}\},$

$S_3^1(i)=\{v \in V\backslash V(C): N_C(v)=\{v_{i-1},v_i,v_{i+1}\}\},$

$S_3^2(i)=\{v \in V\backslash V(C): N_C(v)=\{v_{i-2},v_i,v_{i+2}\}\},$

$S_4(i)=\{v \in V\backslash V(C): N_C(v)=\{v_{i-2},v_{i-1},v_{i+1},v_{i+2}\}\},$

$S_5=\{v \in V\backslash V(C): N_C(v)=V(C)\}.$

Let  $S_3^1=\bigcup_{i=1}^5S_3^1(i)$, $S_3^2=\bigcup_{i=1}^5S_3^2(i)$ and $S_4=\bigcup_{i=1}^5S_4(i)$.
Since $G$ is $P_5$-free, we have that

$$ V(G)=S_0\cup S_2\cup S_3^1\cup S_3^2\cup S_4\cup S_5.$$

We now prove a number of useful properties of these sets using the fact that $G$ is $(P_5,dart)$-free. All properties are proved for $i=1$
due to symmetry.

 \begin{enumerate}[label=\bfseries (\arabic*)]

         \item $S_0$ is anticomplete to $S_2 \cup S_3^1$.\label{s0s2s31}

         Let $u \in S_0$, $v \in S_2(1) \cup S_3^1(1)$. If $uv \in E(G)$, then $\{u,v,v_2,v_3,v_4\}$ is an induced $P_5$.

          \item For each $1 \leq i\leq 5$, $S_3^2(i)$ is not mixed on any edge of $S_0$.\label{s0s32}

          Let $uu'$ be an edge of $S_0$, and $v \in S_3^2(1)$. If $uv \in E(G)$ and $uv' \notin E(G)$, then $\{u',u,v,v_1,v_2\}$ is an induced $P_5$.

          \item $S_0$ is anticomplete to $S_4 \cup S_5$ .\label{s0s4s5}

          Let $u \in S_0$ and $v \in S_4(1) \cup S_5$. If $uv \in E(G)$, then $\{v_3,v_4,v_5,v,u\}$ is a dart.

           \item  Let $A$ be a component of $S_0$, then $A$ is homogeneous and $P_3$-free.\label{s0p3-free}

        By~\ref{s0s2s31}--\ref{s0s4s5}, we have $A$ is homogeneous. Suppose that $A$ contains a $P_3 =uvw$. Since $G$ is connected, there must exist some vertex $x \in N(A)$. By~\ref{s0s2s31}--\ref{s0s4s5},
        $x \in S_3^2(i)$ for some $1 \leq i \leq 5$. Then $\{u,v,w,x,v_i\}$ is a dart. So, $A$ is $P_3$-free.

         \item For each $1 \leq i \leq 5$, $S_3^1(i)$ is not mixed on any edge of $S_2(i)$.\label{s2s311}

         Let $uu'$ be an edge of $S_2(1)$ and $v \in S_3^1(1)$. If $vu' \in E(G)$ and $vu \notin E(G)$, then $\{u,u',v,v_2,v_3\}$ is a dart.

         \item  For each $1 \leq i \leq 5$, $S_2(i)$ is complete to $S_3^1(i+1) \cup S_3^1(i-1)$.\label{s2s312}

         Let $u \in S_2(1), v \in S_3^1(2)$. If $uv \notin E(G)$, then $\{v_1,v,v_3,v_2,u\}$ is a dart. By symmetry, $S_2(1)$ is complete to $S_3^1(5)$.

         \item For each $1 \leq i \leq 5$, $S_2(i)$ is anticomplete to $S_3^1(i+2) \cup S_3^1(i-2)$.\label{s2s313}

         Let $u \in S_2(1), v \in S_3^1(3)$. If $uv \in E(G)$, then $\{v_1,v_5,u,v,v_3\}$ is an induced $P_5$. By symmetry, $S_2(1)$ is anticomplete to $S_3^1(4)$.

         \item For each $1 \leq i \leq 5$, $S_2(i)$ is complete to $S_3^2(i)$.\label{s2s321}

         Let $u \in S_2(1)$ and $v \in S_3^2(1)$. If $uv \notin E(G)$, then $\{u,v_5,v_1,v,v_3\}$ is an induced $P_5$.

         \item For each $1 \leq i \leq 5$, $S_2(i)$ is anticomplete to $S_3^2 \setminus S_3^2(i)$.\label{s2s322}

         Let $u \in S_2(1), v \in S_3^2(2)$. If $uv \in E(G)$, then $\{u,v,v_4,v_5,v_1\}$ is a dart. By symmetry, $S_2(1)$ is anticomplete to $S_3^1(5)$. Let $ w \in S_3^2(3)$. If $uw \in E(G)$, then $\{u,w,v_1,v_5,v_3\}$ is a dart. By symmetry, $S_2(1)$ is anticomplete to $S_3^1(4)$.

         \item For each $1 \leq i \leq 5$, $S_2(i)$ is anticomplete to $S_4(i)$.\label{s2s41}

         Let $u \in S_2(1), v \in S_4(1)$. If $uv \in E(G)$, then $\{u,v,v_3,v_2,v_1\}$ is a dart.

         \item For each $1 \leq i \leq 5$, $S_2(i)$ is complete to $S_4\setminus S_4(i)$.\label{s2s42}

          Let $u \in S_2(1), v \in S_4(2)\cup S_4(3)$. If $uv \notin E(G)$, then $\{v_1,v,v_4,v_5,u\}$ is a dart. By symmetry, $S_2(1)$ is complete to $S_4(5) \cup S_4(4)$.

          \item For each $1 \leq i \leq 5$, $S_2(i)$ is complete to $S_2(i+1)\cup S_2(i-1)$.\label{s2s22}

          Let $u \in S_2(1), v \in S_2(2)$. If $uv \notin E(G)$, then $\{u,v_5,v_1,v,v_3\}$ is an induced $P_5$. By symmetry, $S_2(1)$ is complete to $S_2(5)$.

          \item For each $1 \leq i \leq 5$, $S_2(i+2)\cup S_2(i-2)$ is not mixed on any edge of $S_2(i)$.\label{s2s23}

          Let $uu'$ be an edge of $S_2(1)$ and $v \in S_2(3)$. If $vu' \in E(G)$ and $vu \notin E(G)$, then $\{v_3,v_4,v,u',u\}$ is an induced $P_5$. By symmetry, $S_2(4)$ is not mixed on any edge of $S_2(1)$.

          \item If $S_2 \neq \emptyset$, then $S_5 = \emptyset$. If $S_5 \neq \emptyset$, then $S_2 = \emptyset$.\label{s5ands2}

        Suppose that $S_2 \neq \emptyset$ and $S_5 \neq \emptyset$. Let $u \in S_2(1)$ and $v \in S_5$. If $uv \notin E(G)$, then $\{v_1,v,v_4,v_5, u\}$ is a dart. If $uv \in E(G)$, then $\{u,v_2,v_1,v,v_4\}$ is a dart.

          \item  Let $A$ be a component of $S_2(i)$. Then $A$ is homogeneous and $A$ is $P_3$-free.\label{6-vertex:s2}

        By~\ref{s2s311}--\ref{s5ands2}, we have $A$ is a homogeneous set. Suppose that $A$ contains a $P_3=uvw$, then $\{u,v,w,v_{i-1},v_i\}$
        is a dart. So $A$ is $P_3$-free.

          \item For each  $1 \leq i \leq 5$, $S_5$ is complete to $S_3^2(i)$.\label{s5s32}

          Let $u \in S_5$, $v \in S_3^2(1)$. If $uv \notin E(G)$, then $\{v_2,u,v_5,v_1,v\}$ is a dart.

          \item Let $u,u' \in S_5$ with $uu' \notin E(G)$. Then every vertex in $S_3^1 \cup S_4$ is adjacent to at least one of $u$
          and $u'$.\label{s5s31s4}

          Let $v \in S_3^1(1) \cup S_4(4)$. If $v$ is adjacent to neither $u$ nor $u'$, then $\{u,u',v_4,v_5,v\}$ is a dart.

          \item For $1\le i \le 5$, $S_3^1(i)$, $S_3^2(i)$ and $S_4(i)$ is a clique, respectively.\label{s31s32s4:clique}

           Let  $u,v \in S_3^1(1)$. If $uv \notin E(G)$, then $\{u,v,v_1,v_5,v_4\}$ is a dart. So $S_3^1(1)$ is a clique.
            Let $u,v \in S_3^2(1)$. If $uv \notin E(G)$, then $\{u,v,v_4,v_3,v_2\}$ is a dart. So $S_3^2(1)$ is a clique.
            Let $u,v \in S_4(1)$. If $uv \notin E(G)$, then $\{u,v,v_3,v_2,v_1\}$ is a dart. So $S_4(1)$ is a clique.

         \end{enumerate}

  \section{The proof of \autoref{th:6-vertex-critical}}\label{6-vertex-critical}

\begin{proof}
       We prove the theorem by induction on $k$. If $1 \le k \le 4$, there are finitely many $k$-vertex-critical $(P_5,dart)$-free
       graphs by \autoref{4-vertex-cri}. In the following, we assume that $k \ge 5$ and there are finitely many $i$-vertex-critical
       graphs for $i \le {k-1}$. Now, we consider the case of $k$.

        Let $G=(V,E)$ be a $k$-vertex-critical $(P_5,dart)$-free graph. We show that $|G|$ is bounded. Let $\mathcal{L} = \{K_k,\overline{C_{2k-1}}\}$. If $G$ has a subgraph isomorphic to
  a member $L \in \mathcal{L}$, then $|V(G)|=|V(L)|$ by the definition of vertex-critical and so we are done.
  So, we assume in the following $G$ has no induced subgraph isomorphic to a member in $\mathcal{L}$. Then $G$ is imperfect. Since $G$ is $k$-vertex-critical and $\chi(\overline{C_{2t+1}}) \geq k+1$ if $t \geq k$, it follows that $G$ does not contain $\overline{C_{2t+1}}$ for $t \geq k$.
  Moreover, since $G$ is $P_5$-free, it does not contain $C_{2t+1}$ for $t \geq 3$. It then follows from \autoref{thm:SPGT}, $G$ must contain some $\overline{C_{2t+1}}$ for $2 \le t \le k-2$. To finish the proof, we only need to prove the following two lemmas.
  \subsection{5-hole}

  \begin{lemma}\label{lem:6-vertex:c5}
   If $G$ contains an induced $C_5$, then $G$ has finite order.
  \end{lemma}
  \begin{proof}

  Let $C=v_1,v_2,v_3,v_4,v_5$ be an induced $C_5$. We partition $V(G)$ with respect to $C$.

  Since $G$ is $K_k$-free, combined with ~\ref{s31s32s4:clique}, we have $S_3^1(i)$, $S_3^2(i)$ and $S_4(i)$ is $K_{k-2}$-free, respectively.
Then $|S_3^1(i)| \le k-3$, $|S_3^2(i)| \le k-3$ and $|S_4(i)| \le k-3$. Therefore, $|S_3^1| \le 5k-15$, $|S_3^2| \le 5k-15$ and $|S_4| \le 5k-15$.
Thus, in the following, we only need to bound $S_0$, $S_2$ and $S_5$.

  We first bound $S_0$.

         \begin{claim}\label{6-vertex:s0:homo}
         Let $A$ be a component of $S_0$, then  $\chi(A) \leq k-2.$
         \end{claim}
         \begin{proof}
         Suppose that $\chi(A) \geq k-1$. Since $G$ is connected, there must exist $v \in N(A)$ and $v \in S_3^2$ by~\ref{s0s2s31}--\ref{s0s4s5}. Since $A$ is homogeneous, $v$ is complete to $A$. Then $\chi(G[V(A)\cup \{v\}]) \geq k$. Since $G[V(A)\cup \{v\}]\subset G$, it contradicts with $G$ is $k$-vertex-critical.
        \end{proof}

 Let $A$ be a component of $G$,  we call $A$ a $K_i$-component if $A \cong K_i$ where $i \ge 1$.

         \begin{claim}\label{6-vertex:s0}
         Each component of $S_0$ is a $K_m$ where $1 \le m \le k-2$.
         For each $1 \le m \le k-2$, the number of $K_m$-components is not more than $2^{5k-15}$.
         \end{claim}

         \begin{proof}
          By~\ref{s0p3-free}, $S_0$ is a disjoint  union of cliques. Combined with \autoref{6-vertex:s0:homo}, it follows that each component of $S_0$ is a $K_m$ where $1 \le m \le k-2$.

          Suppose that the number of $K_1$-components in $S_0$ is more than $2^{5k-15} \geq 2^{|S_3^2|}$. The pigeonhole principle shows that
           there are two $K_1$-components $u,v$ having the same neighborhood in $S_3^2$. Since $S_0$ is anticomplete to $S_2 \cup S_3^1\cup S_4 \cup S_5 $. Then $u,v$ have the same neighborhood in $V(G)$. This contradicts with \autoref{lem:XY}.

           Similarly, we can show that the number of $K_2$-components $,\ldots,$ $K_{k-2}$-components is not more than $2^{5k-15}$, respectively.
           \end{proof}

           By \autoref{6-vertex:s0}, it follows that $S_0$ is bounded. Next, we bound $S_2$ and $S_5$. Note that at least one of $S_2$ and
           $S_5$ is an empty set by~\ref{s5ands2}. In the following, we first assume that $S_5 \neq \emptyset$. Then $S_2 = \emptyset$.

       \begin{claim}\label{k-ver:s5}
       $S_5$ is bounded.
       \end{claim}

       \begin{proof}

Let $N_0= S_3^1 \cup S_4$, $N_i = \{v|v \in S_5\setminus {\cup_{j=0}^{i-1}N_j}$, $v $ has a nonneighbor in $N_{i-1}\}$
       where $i \ge 1$.

       If $\overline{G[S_5 \cup S_3^1 \cup S_4]}$ is connected. Since $|S_3^1 \cup S_4| \le 10k-30$ and there exist $w \in V(C)$
       such that $w$ is complete to $S_3^1 \cup S_4 \cup S_5$. By~\ref{s5s31s4}, every vertex in $S_3^1 \cup S_4$ is adjacent to at least one of $w_1$ and $w_2$ where $w_1,w_2 \in S_5$ and $w_1w_2 \notin E(G)$. Then $|S_5|$ is a function of $\chi(S_5)$ by \autoref{ST}.
      Since $\chi(S_5) \le k-3$, it follows that $S_5$ is bounded.

       If $\overline{G[S_5 \cup S_3^1 \cup S_4]}$ is not connected. Then there exists an integer $j \ge 0$ such that $N_0,N_1,\ldots,N_j \neq \emptyset$ but $N_{j+1} = \emptyset$. Then $S_5 - \cup_{i=0}^{j}N_i$ is complete to $\cup_{i=0}^{j}N_i$, and so $S_5 - \cup_{i=0}^{j}N_i$
       is a homogeneous set.
       Since $\chi(S_5 - \cup_{i=0}^{j}N_i) \le k-3$, each component of $S_5 - \cup_{i=0}^{j}N_i$ is an $m$-vertex-critical graph with $1 \le m \le k-3$ by \autoref{lem:homoge}.
        By the inductive hypothesis, it follows that there are finitely many $m$-vertex-critical $(P_5,dart)$-free graphs with $1 \le m \le k-3$. By the pigeonhole principle, the number of each kind of graph is not more than $2$. So, $S_5 - \cup_{i=0}^{j}N_i$ is bounded.
        For each $N_i$ with $1 \le i \le j$, $N_i$ is bounded by \autoref{ST}. Therefore, $S_5$ is bounded.
       \end{proof}

     Thus, $|G|$ is bounded if $S_5 \neq \emptyset$. Next, we assume that $S_2 \neq \emptyset$. Then $S_5 = \emptyset$ by~\ref{s5ands2}.

         \begin{claim}\label{6-vertex:s2:k1k2k3k4}
          Each component of $S_2(i)$ is a $K_m$  where $1 \le m \le k-2$. For each $1 \le m \le k-2$, the number of $K_m$-components is not more than $2^{k-2}+1$.
       \end{claim}

       \begin{proof}
       By~\ref{6-vertex:s2}, $S_2(i)$ is a disjoint  union of cliques. If $S_2(i)$ contains a $K_{k-1}$, then $G$ contains a $K_k$, a contradiction.
       So each component of $S_2(i)$ is a $K_m$ where $1 \le m \le k-2$.

       Suppose that the number of $K_1$-components in $G[S_2(1)]$ is more than $2^{k-2}+1 \geq 2\cdot 2^{|S_3^1(1)|}+1$. By the pigeonhole  principle, there are $y_1,y_2,y_3$ having the same neighborhood  in $S_3^1(1)$. For any  $y_i \neq y_j$ with $i,j \in \{1,2,3\}$, since $y_i$ and $y_j$ are not comparable, there must exist $y_i' \in {N(y_i)}\setminus{N(y_j)}$, $y_j' \in {N(y_j)}\setminus {N(y_i)}$. By~\ref{s2s311}--\ref{s2s23}, $y_i', y_j' \in S_2(3)\cup S_2(4)$.

         If $y_i',y_j' \in S_2(3)$, then $y_i'y_j' \in E(G)$. For otherwise $\{y_i',y_i,v_5,y_j,y_j'\}$ is an induced $P_5$. Then $\{y_i',y_j',y_j,v_2,v_1\}$ is a dart, a contradiction. By symmetry, $y_i',y_j' \notin S_2(4)$. In the following, by symmetry, we assume that $y_i' \in S_2(3)$,
         $y_j' \in S_2(4)$. Then $y_i'y_j' \in E(G)$, otherwise $\{y_j,y_j',v_3,v_4,y_i'\}$ is an induced $P_5$.
         Let  $y_l \neq y_i,y_j$ with $1\leq l \leq 3$. Then $y_ly_i' \notin E(G)$, otherwise $\{y_i,y_i',y_l,v_2,v_1\}$ is a dart. By symmetry, $y_ly_j' \notin E(G)$. Since $y_i$ and $y_l$ are not comparable, there must exist $y_l' \in {N(y_l)}\setminus{N(y_i)}$. Then $y_l' \in S_2(4)$,
         $y_jy_l' \notin E(G)$ and $y_ly_j' \notin E(G)$. Then $y_l'y_j' \in E(G)$, otherwise $\{y_l ,y_l',v_3,y_j',y_j\}$ is an induced $P_5$. Then $\{y_j,y_j',y_l',v_5,v_1\}$ is a dart, a contradiction.

         Similarly, we can show that the number of $K_2$-components $,\dots,$ $K_{k-2}$-components is not more than $2^{k-2}+1$, respectively.
       \end{proof}

          Therefore, $|G|$ is bounded if $S_2 \neq \emptyset$. This completes the proof of \autoref{lem:6-vertex:c5}.
          \end{proof}

       \subsection{$(2t+1)$-antihole}

       \begin{lemma}\label{lem:6-vertex:c7}
        If $G$ contains an induced $\overline{C_{2t+1}}$ for $3 \le t \le k-2$, then $G$ has finite order.
       \end{lemma}

       \begin{proof}
       Let $C = v_1,v_2,\ldots,v_{2t+1}$ be an induced $\overline{C_{2t+1}}$ such that $v_iv_j \in E(G)$ if and only if $|i - j| > 1$.
       All indices are modulo $2t+1$.  We partition $V(G)$ with respect to $C$.
       In the following, we shall write $S_2(v_i,v_{i+1})$ for $S_2(\{v_i,v_{i+1}\})$, $S_3(v_i,v_{i+1},v_{i+3})$ for $S_3(\{v_i,v_{i+1},v_{i+3}\})$, etc.

       \begin{claim}\label{6-cri:s123}
       $S_m = \emptyset$ for $1 \leq m \leq t$.
       \end{claim}
       \begin{proof}
     Let $u \in S_m$ for some  $1 \le m \le t$. Since $u$ has at most $t$ neighbors on $C$,
       there exists an index $i$ such that $u$ is adjacent to neither $v_{i-t}$ nor $v_{i+t}$.

       If $u$ is mixed on $\{v_{i-t+2},v_{i-t+3},\ldots,v_i,v_{i+1},\ldots,v_{i+t-2}\}$.
       Then there exist $j,s \in \{i-t+2,i-t+3,\ldots,i,i+1,\ldots,i+t-2\}$
       with $|j-s|\ge 2$ such that $u$ distinguishes $v_j$ and $v_s$. Then $\{v_{i-t},v_{i+t},v_j,v_s,u\}$ induces a dart.

         If $u$ is complete to $\{v_{i-t+2},v_{i-t+3},\ldots,v_i,v_{i+1},\ldots,v_{i+t-2}\}$, then $u$ is anticomplete to $\{v_{i-t+1},v_{i+t-1}\}$.
        For otherwise $u$ would be complete to $\{v_{i-t+1},v_{i-t+2},\ldots,v_{i+t-1}\}$
       which contradicts the assumption that $u$ has at most $t$ neighbors  on $C$. Then $\{v_{i-t+1},v_{i+t},v_{i+t-1},v_i,u\}$ induces a dart.

       Now assume that $u$ is anticomplete  to
       $V(C) \setminus \{v_{i-t+1},v_{i+t-1}\}$.
       By symmetry, we assume that $u$ is adjacent to $v_{i-t+1}$. Then
       $\{u,v_{i-t+1},v_{i+t},v_{i-t+2},v_{i-t}\}$
       is an induced $P_5$.

       Therefore, $S_m = \emptyset$, for $1 \le m \le t$.
\end{proof}

       \begin{claim}\label{6-cri:s0}
       $S_0 = \emptyset$.
       \end{claim}

       \begin{proof}
       Suppose not. Let $u \in S_0$. By \autoref{6-cri:s123} and the connectivity of $G$, $u$ has a neighbor $u' \in S_i$
       for some $t+1 \le i \le 2t+1$. Then there exists an index $i$ such that $u'$ is adjacent to $v_{i-t}$, $v_{i+t}$ and some vertex
       $v_j$ where $j \in \{i-t+2,\ldots,i,i+1,\ldots,i+t-2\}$. Now $\{v_{i-t},v_{i+t},v_j,u',u\}$ induces a dart. This proves $S_0 = \emptyset$.
       \end{proof}

       \begin{claim}\label{6-cri:s456}
       For every $X \subseteq C$ with $t+1 \le |X| \le 2t$, $S(X)$ is a clique.
       \end{claim}

       \begin{proof}
       Let $X \subseteq C$ be an arbitrary set with $t+1 \le |X| \le 2t$. Since $t+1 \le |X| \le 2t$, there exists an index $i \in C$
       such that $v_i \in X$ but $v_{i+1} \notin X$, and $v_j \in X$ where $v_j \in V(C)\setminus \{v_{i-2},v_{i-1},\ldots,v_{i+2}\}$. If $S(X)$ contains two non-adjacent
       vertices $u$ and $u'$, then $\{u,u',v_i,v_j,v_{i+1}\}$ induces a dart. Therefore, $S(X)$ is a clique.
       \end{proof}

       Since $G$ is $K_k$-free, by \autoref{6-cri:s456}, it follows that $|S(X)| \le k-3$ for $X \subseteq C$
with $t+1\le |X| \le 2t$. Therefore, $|S_{t+1} \cup \ldots \cup S_{2t}| \le (k-3)(\tbinom{2t+1}{t+1} + \dots +\tbinom{2t+1}{2t})$.
In the following, we  bound $S_{2t+1}$.

       \begin{claim}\label{s2t+1}
          Let $u,u' \in S_{2t+1}$ with $uu' \notin E(G)$.
          Then every vertex in $S_{t+1} \cup \ldots \cup S_{2t}$ is adjacent to at least one of $u$ and $u'$.
          \end{claim}

          \begin{proof}
          Let $v \in S_m$ with $t+1 \le m \le 2t$. Then there exist two
          vertices $v_i, v_j \in V(C)$ with $|i-j| \ge 2$ such that $v$  distinguishes $v_i$ and $v_j$.
          If $v$ is adjacent to neither $u$ nor $u'$, then $\{u,u',v_i,v_j,v\}$ induces
          a dart.
          \end{proof}

          \begin{claim}\label{k-ver:s2t+1}
       $S_{2t+1}$ is bounded.
       \end{claim}

       \begin{proof}

       Let $N_0= S_{t+1} \cup \ldots \cup S_{2t}$, $N_i = \{v|v \in S_{2t+1}\setminus {\cup_{j=0}^{i-1}N_j}$, $v $ has a nonneighbor in $N_{i-1}\}$
       where $i \ge 1$.

        If $\overline{G[S_{t+1} \cup \ldots \cup S_{2t+1}]}$ is connected. Since $|S_{t+1} \cup \ldots \cup S_{2t}| \le (k-3)(\tbinom{2t+1}{t+1} + \dots +\tbinom{2t+1}{2t})$ and there exists $w \in V(C)$ such that $w$ is complete to $S_{t+1} \cup \ldots \cup S_{2t} \cup S_{2t+1}$.
       By \autoref{k-ver:s2t+1}, every vertex in $S_{t+1} \cup \ldots \cup S_{2t}$ is adjacent to at least one of $w_1$ and $w_2$ where $w_1,w_2 \in S_{2t+1}$ and $w_1w_2 \notin E(G)$. Then $|S_{2t+1}|$ is a function of $\chi(S_{2t+1})$ by \autoref{ST}. Since $\chi(S_{2t+1}) \le k-t-1 $,  it follows that $S_{2t+1}$ is bounded.

       If $\overline{G[S_{t+1} \cup \ldots \cup S_{2t+1}]}$ is not connected. Then there exists an integer $j \ge 0$ such that $N_1,N_2,\ldots,N_j \neq \emptyset$ but $N_{j+1} = \emptyset$. Then $S_{2t+1} - \cup_{i=0}^{j}N_i$ is complete to $\cup_{i=0}^{j}N_i$, and so $S_{2t+1} - \cup_{i=0}^{j}N_i$
       is a homogeneous set.
       Since $\chi(S_{2t+1} - \cup_{i=0}^{j}N_i) \le k-t-1$, each component of $S_{2t+1} - \cup_{i=0}^{j}N_i$ is an $m$-vertex-critical graph with $1 \le m \le k-4$ by \autoref{lem:homoge}.
       By the inductive hypothesis, it follows  that there are finitely many $m$-vertex-critical$(P_5,dart)$-free graph for $1 \le m \le k-4$. By the pigeonhole principle, the number of each kind of graph is not more than $2^{|S_{t+1} \cup \ldots \cup S_{2t}|}$. So, $S_{2t+1} - \cup_{i=0}^{j}N_i$ is bounded. For each $N_i$ with $1 \le i \le j$, $N_i$ is bounded by \autoref{ST}. Therefore, $S_{2t+1}$ is bounded.
\end{proof}

This completes the proof of \autoref{lem:6-vertex:c7}.
 \end{proof}

By \autoref{lem:6-vertex:c5}--\autoref{lem:6-vertex:c7}, it follows that \autoref{th:6-vertex-critical} holds.
\end{proof}

\section{Complete characterization for $k \in \{5,6,7\}$}\label{algorithmSection}

In Section~\ref{6-vertex-critical}, we proved that there are finitely many $k$-vertex-critical $(P_5,dart)$-free graphs by showing the existence of an upper bound for the order of such graphs for every integer $k \geq 1$. These bounds are not necessarily sharp. In the current section, we show sharp upper bounds for $k \in \{5,6,7\}$ by computationally determining an exhaustive list of all $k$-vertex-critical $(P_5,dart)$-free graphs.

We created two independent implementations of the algorithm from Goedgebeur and Schaudt~\cite{GS18}, which we extended for the $(P_5,dart)$-free case. The source code of these implementations can be downloaded from~\cite{criticalpfree-site} and~\cite{jorik-github}. The algorithm expects three parameters as an input: an integer $k \ge 1$, a set of graphs $\mathcal{H}$ and a graph $I$. It generates all $k$-vertex-critical $\mathcal{H}$-free graphs that contain $I$ as an induced subgraph. The pseudocode is given in \autoref{algo:extend}. The algorithm is not guaranteed to terminate (e.g.\ there could be infinitely many such graphs). However, if the algorithm terminates, it is guaranteed that the generated graphs are exhaustive. The algorithm works by recursively extending a graph with one vertex and adding edges between this new vertex and already existing vertices. In each step of the recursion, the algorithm uses powerful pruning rules that restrict the ways in which the edges are added. These pruning rules allow the algorithm to terminate some branches of the recursion and are the reason why the algorithm itself can terminate in some cases. For example, a graph which is not $(k-1)$-colorable cannot appear as a proper induced subgraph of a $k$-vertex-critical graph, so the algorithm does not need to extend such graphs. In general, the pruning rules that are used by the algorithm are much more sophisticated than this. We refer the interested reader to~\cite{GS18} for further details about the correctness of the algorithm and the different pruning rules.

\begin{algorithm}[ht!]
\caption{Extend(An integer $k$, A set of graphs $\mathcal{H}$, A graph $I$)}
\label{algo:extend}
  \begin{algorithmic}[1]
		\IF{$I$ is $\mathcal{H}$-free AND not generated before}
			\IF{$I$ is not $(k-1)$-colorable}
				\IF{$I$ is a $k$-vertex-critical graph}
					\STATE output $I$
				\ENDIF	
			\ELSE
				\FOR{every graph $I'$ obtained by adding a new vertex $u$ to $I$ and edges between $u$ and vertices in $V(I)$ in all possible ways that are permitted by the pruning rules}
						\STATE Extend($k$,$\mathcal{H}$,$I'$)
				\ENDFOR					
			\ENDIF	
		\ENDIF	
  \end{algorithmic}
\end{algorithm}

We now prove the following characterization theorem:

\begin{theorem}\label{th:enumeration}
   There are exactly 184 5-vertex-critical $(P_5,dart)$-free graphs and the largest such graphs have order 13. There are exactly 18,029 6-vertex-critical $(P_5,dart)$-free graphs and the largest such graphs have order 16. There are exactly 6,367,701 7-vertex-critical $(P_5,dart)$-free graphs and the largest such graphs have order 19.
\end{theorem}
\begin{proof}
We saw in Section~\ref{6-vertex-critical} that every $k$-vertex-critical $(P_5,dart)$-free graph is either $K_k$, $\overline{C_{2k-1}}$ or contains $C_5$ as an induced subgraph or contains $\overline{C_{2t+1}}$ as an induced subgraph for some $2 \le t \le k-2$. If \autoref{algo:extend} is called with the parameters $k \in \{5,6,7\}$, $\mathcal{H}=\{P_5,dart\}$ and $I$ one of these graphs, the algorithm terminates in less than a second for $k=5$, less than a minute for $k=6$ and a few hours for $k=7$ (for all choices of $k$ and $I$). The counts of these graphs are reported in Table~\ref{table:counts}. The results of the two independent implementations of this algorithm (cf.\ \cite{criticalpfree-site} and~\cite{jorik-github}) are in complete agreement with each other.
\end{proof}

Table~\ref{table:counts} gives an overview of the number of $k$-critical and $k$-vertex-critical $(P_5,dart)$-free graphs for $k \in \{5,6,7\}$. A graph $G$ is $k$-critical $(P_5,dart)$-free if it is $k$-chromatic, $(P_5,dart)$-free and every $(P_5,dart)$-free proper subgraph of $G$ is $(k - 1)$-colorable.
The graphs from Table~\ref{table:counts} can be obtained from the meta-directory of the \textit{House of Graphs}~\cite{CDG} at \url{https://houseofgraphs.org/meta-directory/critical-h-free}. Moreover, the $k$-critical graphs from Table~\ref{table:counts} can also be inspected in the searchable database of the \textit{House of Graphs}~\cite{CDG} by searching for the keywords ``critical (P5,dart)-free''.
The 5-critical $(P_5,dart)$-free graphs are shown in \autoref{fig:5CritGraphs}.

\begin{table}[ht!]
    \centering
    \setlength{\tabcolsep}{4pt}
    \begin{tabular}{| l | *{7}{c} | c |}
        \hline
        Vertices           & 5& 6     & 7     & 8     & 9     & 10    & 11            & Total        \\ \hline
        5-critical         & 1 & & 1  & 1  & 7    & 1     &       &    \\
        5-vertex-critical  & 1 &  & 1   & 6   & 172     & 1 &       &    \\
        6-critical          & & 1 &   & 1  & 1    &  6    & 33      &   \\
        6-vertex-critical   & & 1  &  & 1   & 6     & 171      & 17,834       &    \\
        7-critical          &  &  & 1  &   & 1 & 1 & 6 &    \\
        7-vertex-critical   & &   & 1   &    & 1     & 6      & 171       &    \\ \hhline{|=|=======|=|}
        Vertices                   & 12    & 13    & 14    & 15    & 16    & 19           &  &            \\ \hline
        5-critical          & & 3 &   &   &     &      &       &   14  \\
        5-vertex-critical   & & 3  &   &    &      &       &       &   184 \\
        6-critical          & 2 & 1 &   &   &  13  &      &       &  58 \\
        6-vertex-critical   & 2 &  1 &   &   & 13     &      &       &  18,029  \\
        7-critical          & 28 & 250 & 6 & 2 & 1 & 35  & & 331 \\
        7-vertex-critical   & 17,834  & 6,349,644   & 6   & 2     & 1      & 35   &    & 6,367,701    \\
        \hline
    \end{tabular}
    \caption{The number of $k$-critical and $k$-vertex-critical $(P_5,dart)$-free graphs (for $k \in \{5,6,7\}$).}
    \label{table:counts}
\end{table}

\begin{figure}[h!t]
\centering
\includegraphics[width=.22\textwidth]{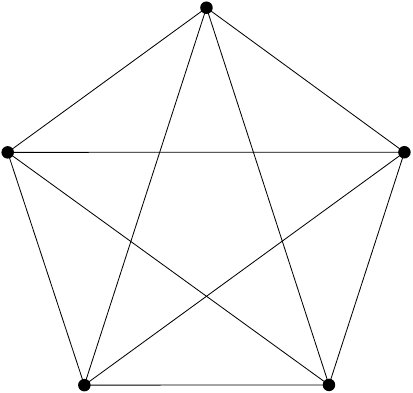}\ \
\includegraphics[width=.22\textwidth]{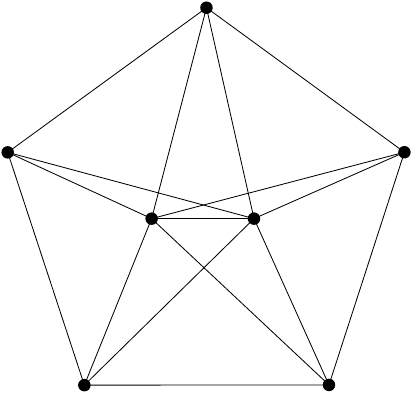} \ \
\includegraphics[width=.22\textwidth]{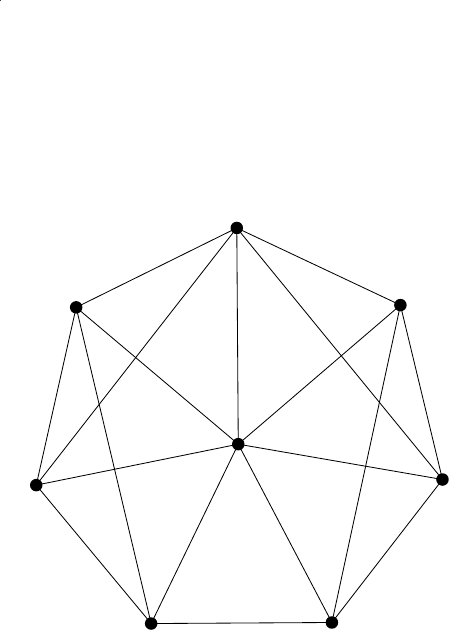}\ \

\smallskip

\includegraphics[width=.22\textwidth]{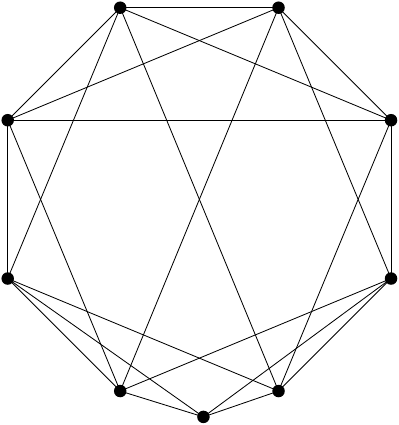} \ \
\includegraphics[width=.22\textwidth]{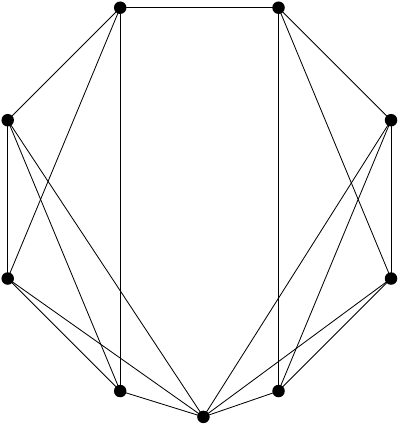}
\includegraphics[width=.22\textwidth]{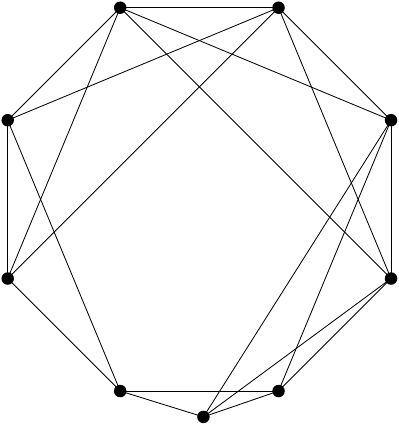}\ \
\includegraphics[width=.22\textwidth]{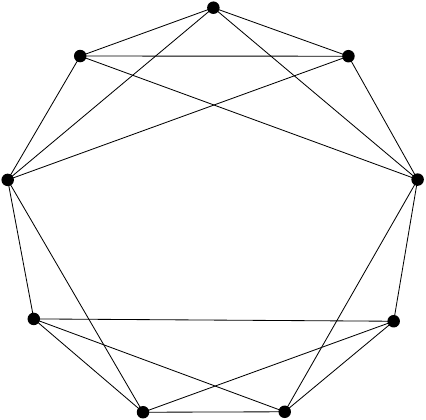} \ \

\smallskip

\includegraphics[width=.22\textwidth]{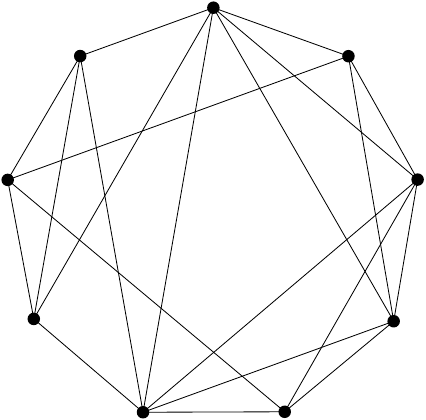}\ \
\includegraphics[width=.22\textwidth]{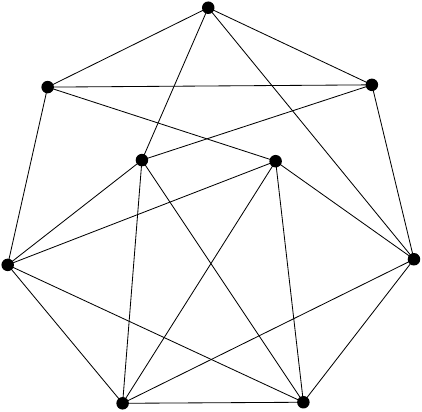} \ \
\includegraphics[width=.22\textwidth]{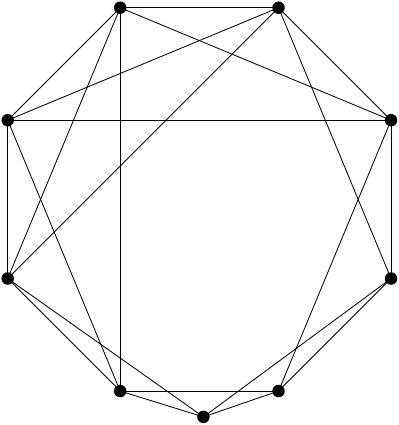}
\includegraphics[width=.22\textwidth]{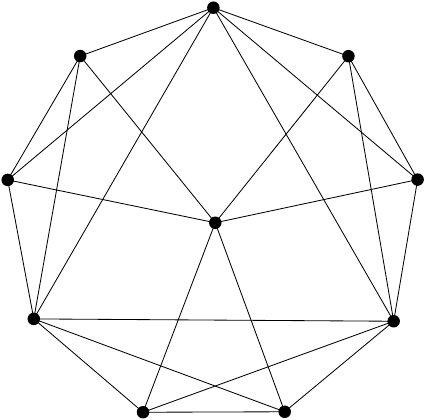}\ \

\smallskip

\includegraphics[width=.22\textwidth]{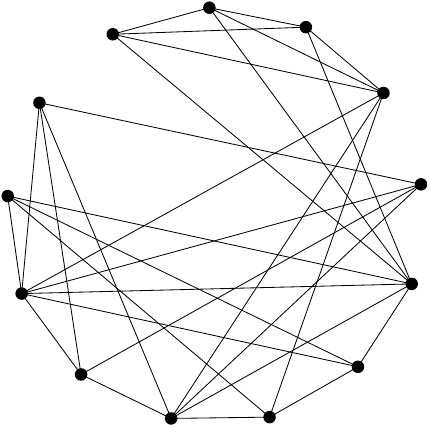} \ \
\includegraphics[width=.22\textwidth]{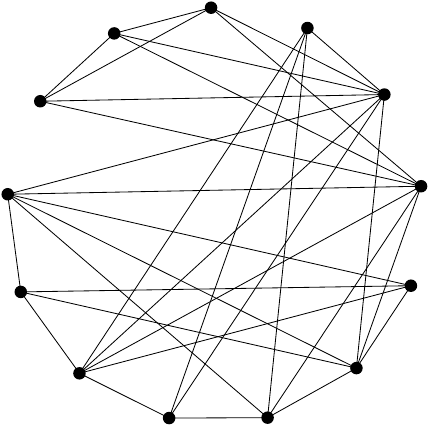}\ \
\includegraphics[width=.22\textwidth]{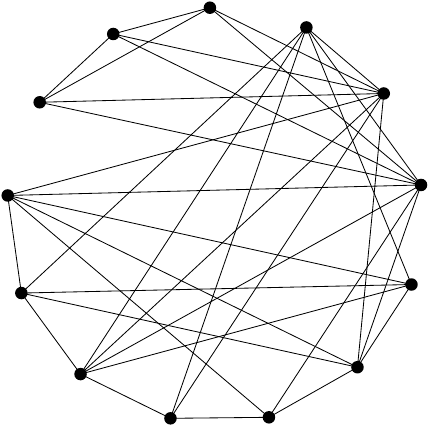} \ \

\caption{All 14 5-critical $(P_5,dart)$-free graphs.}
\label{fig:5CritGraphs}

\end{figure}

\section{Conclusion}\label{conclusion}
In this paper, we have proved that there are finitely many $k$-vertex-critical $(P_5,dart)$-free graphs for $k \ge 1$ and computationally determined an exhaustive list of such graphs for $k \in \{5,6,7\}$. Our results gave an affirmative
answer to the problem posed in~\cite{CGHS21} for $H =dart$.
In the future, it is natural to investigate the finiteness of the set of $k$-vertex-critical $(P_5,H)$-free graphs for other graphs $H$ of order 5, see \cite{CH23}.


\subsection*{Acknowledgements}

The research of Jan Goedgebeur was supported by Internal Funds of KU Leuven. Jorik Jooken is supported by a Postdoctoral Fellowship of the Research Foundation Flanders (FWO). Shenwei Huang is supported by National Natural Science Foundation of China (12171256).

\end{document}